\documentclass{amsproc}

\usepackage{amsmath,amsthm,amssymb,enumitem,tikz-cd}
\usepackage[colorlinks=true]{hyperref}
\hypersetup{linkcolor=blue, urlcolor=blue, citecolor=blue}
\usepackage{cleveref}
\usepackage{verbatim}

\newtheorem{theorem}{Theorem}[section]
\newtheorem{lemma}[theorem]{Lemma}

\theoremstyle{definition}
\newtheorem{definition}[theorem]{Definition}
\newtheorem{example}[theorem]{Example}

\newtheorem{corollary}[theorem]{Corollary}

\theoremstyle{remark}
\newtheorem{remark}[theorem]{Remark}

\numberwithin{equation}{section}

\newcommand{\NN}{\mathbf{N}}

\newcommand{\fm}{\mathfrak{m}}
\newcommand{\fb}{\mathfrak{b}}
\newcommand{\fp}{\mathfrak{p}}
\newcommand{\fq}{\mathfrak{q}}
\newcommand{\Ann}{\operatorname{Ann}}

\newcommand{\Spec}{\operatorname{Spec}}

\newcommand{\hgt}{\operatorname{ht}}
\newcommand{\Min}{\operatorname{Min}}
\newcommand{\fbp}[1]{\left[ #1 \right]}
\newcommand{\fte}{\operatorname{Fte}}
\newcommand{\fdp}{F\text{-depth}\,}
\newcommand{\gfdp}{gF\text{-depth}\,}
\newcommand{\findim}{\operatorname{findim}}
\newcommand{\HSL}{\operatorname{HSL}}

\newcommand{\fQ}{\mathfrak{Q}}

\newcommand{\WN}{\textrm{WN}}

\begin{document}

\title{}


%    author one information
\author{Kyle Maddox}
\address{}
\curraddr{}
\email{}
\thanks{}

%    author two information
\author{Lance Edward Miller}
\address{}
\curraddr{}
\email{}
\thanks{}

%    The 2020 edition of the Mathematics Subject Classification is
%    the current definitive version.
\subjclass[2020]{Primary 13A35, 13D45}

\title{\texorpdfstring{$F$}{F}-depth and \texorpdfstring{$F$}{F}-nilpotent rings: generalizations and applications}
%\author{Kyle Maddox and Lance Edward Miller}

\begin{abstract}
Computation of the Frobenius closure of ideals in rings of prime characteristic is a difficult problem. For a given ideal, its Frobenius test exponent provides a valuable degree of uniformity in performing the calculation. Hence, it is desirable to find uniform upper bounds on the Frobenius test exponent of ideals. Even in nice rings of low dimension, Brenner showed considering the collection of all ideals is generally hopeless, but for Cohen-Macaulay rings, Katzman-Sharp there are uniform upper bounds on the Frobenius test exponent for the class of parameter ideals. Subsequent efforts in controlling the Frobenius test exponents have typically involved studying the Frobenius action on local cohomology and the degree to which this action is nilpotent. 

The goal of this survey article is to examine the history of the Frobenius test exponent problem and its relationships to singularity types for local rings defined in terms of the Frobenius action on local cohomology. We also explore several generalizations of these prior results to a setting where less nilpotence in the Frobenius action is assumed. 

\end{abstract}

\maketitle
%\tableofcontents

\section{Introduction}

Throughout, we consider excellent noetherian rings of prime characteristic $p>0$. Over the last five years, there as been a flurry of research concerning a family of singularity classes related to $F$-nilpotent singularities. These singularities have a relatively short history, introduced formally in \cite{ST17}, though the concept appears earlier \cite{BB05,Lyu97} in different guises. In this short time, these singularities and their variants have occupied an indispensable place in the computation of Frobenius closures of parameter ideals $\fq$, which is the set of elements $\fq^F$ for which $F^e(x) \in F^e(\fq)R$ for some $e$. In a noetherian ring, there is an $e$ such that $x \in \fq^F$ if and only if $x^{p^e} \in F^e(\fq)R$, the least such $e$ is called the \textit{Frobenius test exponent} of $\fq$, written $\fte \fq$. As $\fq$ varies, so might $\fte \fq$, and so it is desirable to have uniform upper bounds on the Frobenius test exponent independent of $\fq$. These upper bounds were first shown to exist by Katzman-Sharp \cite{KS06} and Huneke-Katzman-Sharp-Yao \cite{HKSY06} in the setting of Cohen-Macaulay and generalized Cohen-Macaulay rings respectively. Later work by Quy \cite{Quy19} and Maddox \cite{Mad19} expanded these results to $F$-nilpotent rings and generalizations thereof, which are types of local rings defined in terms of the Frobenius action on local cohomology. A number of recent articles have studied both these singularity types and the Frobenius test exponent problem; see \cite{CMM23,DMP24,HQ19,HQ20,HQ23,KMPS23,MM24,MP23,MS25,PQ19}. The goal of this survey is to give a guided tour through this rapidly advancing literature to help facilitate researchers interested in entering the subject. 

We also aim to expand the main results of several of these articles using the technology of the generalized $F$-depth with respect to an ideal introduced in \cite{CMM23}. Lyubeznik introduced the \textit{$F$-depth} of a local ring in \cite{Lyu06}, generalizing ideas from \cite{HS77}, as the least $j$ for which $H^j_\fm(R)$ does not vanish under any iterate of its natural Frobenius action. $F$-depth plays a role similar to depth in the study of these nilpotent singularity types.

\subsection{\texorpdfstring{$F$}{F}-nilpotent singularities} The primary singularities considered here are variants of $F$-nilpotent singularities as introduced by Blickle-Bondu in \cite{BB05} under the name ``close to $F$-rational". The next major work to feature them and study their properties directly was \cite{ST17}, where Srinivas-Takagi gave them the name $F$-nilpotent. We assume here the audience has some familiarity with the singularities that arise naturally from the Frobenius map, sometimes called $F$-singularities; we refer to \cite{MP} or \cite{SS} for a detailed review. Our first aim is to place $F$-nilpotent singularities in comparison to other $F$-singularities that are more well understood. In the main body of the survey, we revisit the subject of $F$-nilpotent singularities (see \Cref{sec:Fnil}).

Like many other $F$-singularities, the $F$-nilpotent condition is defined in terms of how the Frobenius acts on local cohomology. Let $(R,\fm)$ be a local ring of dimension $d$, and recall the Frobenius map $F:R\rightarrow R$ by $r\mapsto r^p$ acts on local cohomology modules $H_\fm^j(R)$. By minor abuse of notation, we write $F : H_\fm^j(R) \to H_\fm^j(R)$ for the additive map induced by the Frobenius map on $R$. A class $\eta \in H_\fm^j(R)$ is {\it nilpotent} provided $F^e(\eta) = 0$ for some $e$. The class of $F$-nilpotent rings are those for which the local cohomology is ``as nilpotent as possible'' under $F$. Specifically, one can ask, for fixed $j < d$, that every element of $H_\fm^j(R)$ is nilpotent. When this happens for all $j < d$, $R$ is called {\it weakly $F$-nilpotent}. The Frobenius map cannot be nilpotent on $H^{d}_\fm(R)$, and the submodule of elements of $H^d_\fm(R)$ nilpotent under Frobenius is inside the tight closure of $0$ (see \Cref{subsec: frob action} for definitions). We say $R$ is {\it $F$-nilpotent} if it is weakly $F$-nilpotent and every element of the tight closure of $0$ in $H^{d}_\fm(R)$ is nilpotent.

In contrast to $F$-nilpotent singularities, we have the more well-known $F$-injective singularities. These are the singularities for which $F:H^j_\fm(R) \rightarrow H^j_\fm(R)$ is injective for all $0 \le j \le \dim R$. Of course, when Frobenius acts both injectively on $H_\fm^j(R)$ and $H_\fm^j(R)$ is also nilpotent under Frobenius, we must have that $H_\fm^j(R)$ vanishes. Thus it is immediate to see that weakly $F$-nilpotent $F$-injective rings are Cohen-Macaulay. Additionally, $F$-rational singularities are precisely those that are both $F$-nilpotent and $F$-injective.

\begin{center}
\begin{tikzcd}%[column sep=huge]
 %      \text{A}\arrow[r,bend left=10, Rightarrow,"\text{\cref{thm: CMFI-pun}}" {yshift=0.5in}]& \text{B}\arrow[l,bend left=10,sloped,Rightarrow,"+F\text{-injective}"]\arrow[l,bend left=10,sloped,Rightarrow,"\text{\cref{rmk: f-inj  + buchs = f-buchs}}" {yshift=-0.5in}] \\
 \text{F-rational} \arrow[r,Rightarrow] & \text{F-injective} \arrow[l,bend left=30,Rightarrow,"\text{+F-nilpotent}" {yshift=-0.1in}]
\end{tikzcd}
\end{center}

The $F$-nilpotent singularities were considered by Blickle-Bondu \cite{BB05} as ``close to $F$-rational'' and arose naturally in the study of Lyubeznik numbers. However, Srinivas-Takagi take an approach to study $F$-nilpotent singularities via Hodge theory. Specifically, they conjecture that complex singularities which reduce to $F$-nilpotent singularities for all but finitely many $p$ satisfies certain vanishing in the Hodge filtration \cite[Conj. $H_n$]{ST17}. The story here goes deeper, which relates Conjecture $H_n$ to an arithmetic variant \cite[Conj. $N_n$]{ST17} that is closely related to the Weak-Ordinarity conjecture, and thus to relationship between $F$-injective singularities and Du Bois singularities \cite{BST17}. This highlights that $F$-nilpotent singularities are an interesting and important class of $F$-singularities.

The class of weakly $F$-nilpotent singularities are less well understood, but arises naturally with applications in the finiteness problem for Frobenius test exponents. We note however it is somewhat more straightforward to motivate weakly $F$-nilpotent singularities via the analogy to Cohen-Macaulayness. In some sense, these appeared even earlier in that Lyubeznik's theory of $F$-depth can be used to define weakly $F$-nilpotent local rings as those for which the $F$-depth is equal to the dimension.

\subsection{Finiteness of Frobenius test exponents} 

The Frobenius closure of an ideal $I$ in a ring $R$ of prime characteristic $p>0$ is the ideal $I^F$ of elements $x$ for which $x^{p^e} \in F^e(I)R$. If $R$ is a noetherian regular ring, then flatness of the Frobenius map implies $I^F=I$ for all ideals $I$, but in singular rings one can have nontrivial Frobenius closure relations. Hence, computation of Frobenius closure of ideals is interesting from a geometric standpoint, since it relates to singularities of the ring.

Given an ideal $I$ in a noetherian ring $R$, there is an $e$ so that $I^F = \{ x \in R \mid x^{p^e} \in F^e(I)R\}$. The least such $e$ is called the \textbf{Frobenius test exponent} of $I$, written $\fte I$. If $e_0 = \fte I$, then $I^F$ is the preimage under $F^e$ of $F^e(I)R$ for all $e \ge e_0$. Thus, for computational purposes, it is desirable to know upper bounds on $\fte I$, and one can even hope for uniform upper bounds on $\fte I$ over all ideals $I$ in the ring $R$. However, Brenner showed in \cite{Bre06} that $\{\fte I \mid I \subset R \}$ can be an unbounded set even in nice rings of low dimension. 

Due to natural connections between parameter ideals (that is, ideals generated by a system of parameters) and local cohomology modules, Katzman-Sharp showed in \cite[Thm. 2.4]{KS06} that Cohen-Macaulay local rings have a uniform upper bound on the Frobenius test exponent of parameter ideals in terms of the Hartshorne-Speiser-Lyubeznik number of $R$ (see \Cref{sec:HSL}). For this reason, we say the \textbf{Frobenius test exponent} of the ring is $\fte R = \sup\{\fte \fq \mid \fq \subset R \text{ is a parameter ideal}\}$. Outside the Cohen-Macaulay setting, results bounding $\fte R$ typically assume the lower local cohomology modules of $R$ are ``close" to being nilpotent under the canonical Frobenius action on local cohomology. For a history of results since Katzman-Sharp in bounding Frobenius test exponents, see \Cref{sec:genfte}.  

Frobenius test exponents also have applications outside of the computation of Frobenius closure. One common application is in the area of multiplicity, where first Huneke-Watanabe showed in \cite{HW15} that $F$-pure local rings of dimension $d$ have Hilbert-Samuel multiplicity bounded above by $\binom{v}{d}$, where $v$ is the embedding dimension of $R$. Later, Katzman-Zhang showed in \cite{KZ19} that the same bound holds under the weaker assumption that $R$ is $F$-injective and generalized Cohen-Macaulay, and for a reduced Cohen-Macaulay ring that the Hilbert-Samuel multiplicity of the ring is $C\binom{v}{d}$, where $C = (p^{\HSL R})^{v-d}$. Finally, Huong-Quy showed in \cite{HQ20} that, for a local ring $R$ with finite Frobenius test exponent $\fte R = e_0$, that the Hilbert-Samuel mutiplicity of $R$ is bounded by $(p^{e_0})^{v-d}\binom{v}{d}$.

The general question of bounding $\fte I$ over all ideals $I$ has enjoyed fewer positive results. Recently, Maddox-Singh showed in \cite{MS25} that, for an affine semigroup ring defined over a field of prime characteristic, the Frobenius test exponent of all ideals has a uniform upper bound given in terms of the combinatorics of the underlying semigroup.

\subsection{Summary of the main results}  We now summarize the new results in the survey. Our perspective is expanding prior results by substituting $F$-depth for the \textit{generalized $F$-depth with respect to $J$}, denoted $\gfdp_J(R)$,  introduced in \cite{CMM23}. This depth-like invariant captures local information about the $F$-depth on the open set $\Spec R \setminus V(J)$. In particular, we prove the following new result, which generalizes \cite[Prop. 4.6]{KMPS23}. 

\begin{theorem}[\Cref{thm:replacepuncspec}]
    Let $(R,\fm)$ be an $F$-finite, equidimensional local ring of prime characteristic $p>0$ and dimension $d$, and let $J\subset R$ be an ideal. Then, $\gfdp_J(R) \ge t$ if and only if $\fdp(R_\fp) \ge t-d+\hgt \fp$ for all $\fp \in \Spec R \setminus V(J)$.
\end{theorem}

In the spirit of the primary applications considered in the introduction, we consider an effect the generalized $F$-depth with respect to an ideal has on finiteness of ideals generated by filter regular sequences (see \Cref{sec:genfte} for the definition), which are slightly more general than partial systems of parameters. Specifically, we improve on a result \cite[Thm. 1.1]{HQ22} in the $F$-finite setting.

\begin{theorem}[\Cref{thm: gfdp_J and filter regular seqeunces}] Let $(R,\fm)$ be an $F$-finite local ring of dimension $d$ and $J \subset R$ an ideal. Suppose $t \leq \dim R$. Consider the set $$\fQ_{J,t} := \{ (x_1,\ldots,x_t) \subset J \mid x_1,\ldots,x_t \textrm{ is a filter regular sequence} \}.$$ If $\gfdp_J(R) \geq t$, the supremum $\sup\{\fte \fq \mid \fq \in \fQ_{J,t}\}$ is finite. \end{theorem}

Finally, we add two new openness results. Specifically, \cite[Thm. 5.2]{KMPS23} shows the weakly $F$-nilpotent and the $F$-nilpotent loci are open in $F$-finite rings. This can be adapted to the locus of primes for which $\gfdp_J$ is bounded below, using the results from \cite{CMM23}.

\begin{theorem}[\Cref{thm:gfdpopen}]
    Let $(R,\fm)$ be an $F$-finite, equidimensional local ring of dimension $d$, $J\subset R$ an ideal, and $t \ge 0$. The locus \[\{\fp \in \Spec R \mid \gfdp_{JR_\fp}(R_\fp) \ge t \}\] is open.
\end{theorem}

The second is slightly more nuanced but provides further applications to finiteness of Frobenius test exponents. In particular, in some situations, an open locus where $\fte R_\fp$ is uniformly bounded can be constructed based on a comparison along finite, purely inseparable ring extensions.

\begin{theorem}[\Cref{FTE:open}]
Suppose $\phi \colon R\rightarrow S$ is a finite ring extension of $F$-finite rings of prime characteristic $p>0$. Let $X_\phi$ be the set of primes of $\Spec R$ for which $\phi_\fp$ is a purely inseparable extension. There is an $e_0 \in \NN$ such that for all $\fp$ in $X_\phi$ such that $S_\fq$ is $F$-pure, $\fte I \le e_0$ for all ideals $I\subset R_\fp$. In particular, if the image of the $F$-pure locus of $S$ is also open in $\Spec R$, then there is an open set in $\Spec R$ on which the Frobenius test exponent of all ideals has a uniform upper bound.\end{theorem}

{\bf Acknowledgments:} We thank William D. Taylor for comments and corrections from an advance reading of the survey.

\section{Preliminaries}

Throughout, all rings will be noetherian of prime characteristic $p>0$, and for simplicity of exposition, we will also assume all rings are $F$-finite (that is, the Frobenius map is a finite ring map) unless otherwise stated. We write $\lambda_R(M)$ for the length of an $R$-module $M$. When $R$ is local, we reserve $\fq$ for a \textit{parameter ideal} of $R$, that is, an ideal generated by a part of a system of parameters, and say $\fq$ is a \textit{full} parameter ideal if its height is equal to the dimension of the ring. When clarity is needed we will refer to a \textit{partial} parameter ideal for a parameter ideal. The set $\Min R$ (respectively $\operatorname{Ass} R$) denotes the minimal (resp. associated) primes of $R$ and $R^\circ := R\setminus \bigcup_{\fp \in \Min R} \fp.$ Given an ideal $J\subset R$, we write $V(J)$ for the set of prime ideals of $R$ which contain $J$.

\subsection{ Frobenius and ideal closures} Given an $R$-module $M$, we write $F^e_*(M)$ for the restriction of scalars along the iterated Frobenius map $F^e$. For an ideal $I$ in $R$, $I^{\fbp{p^e}}$ denotes the \textit{$e$th Frobenius bracket power of $I$} i.e., the ideal generated by $F^e(I)$ in $R$.

\begin{definition}\label{dff: tight and frobenius closure}
Fix an ideal $I\subset R$. The \textbf{Frobenius closure} of $I$, written $I^F$, is the ideal of $R$ given by \[I^F := \left\lbrace x \in R \left| x^{p^e} \in I^{\fbp{p^e}} \text{ for some } e \gg 0 \right. \right\rbrace\] and the \textbf{tight closure} of $I$, written $I^*$, is the ideal of $R$ given by \[I^* := \left\lbrace x \in R \left| \text{ there is a } c \in R^\circ \text{ such that } cx^{p^e} \in I^{\fbp{p^e}} \text{ for all } e \gg 0 \right. \right\rbrace.\] 
\end{definition}

\noindent It is easy to see that $I\subset I^F \subset I^* \subset \overline{I}$ where $\overline{I}$ denotes the integral closure of $I$. However, if $R$ is a regular ring, $I^*=I$ for all ideals $I$. We will focus primarily on the Frobenius closure of ideals. 

\begin{definition}\label{dff: Fte}
Note that, as $R$ is noetherian, there is an $e$ so that $(I^F)^{\fbp{p^e}} = I^{\fbp{p^e}}$. We call the smallest such $e$ the \textbf{Frobenius test exponent} of $I$, written $\fte I$. 
\end{definition}

When attempting to compute $I^F$ for a specific ring $R$ and ideal $I$, knowing $e_0=\fte I$ is important, since $I^F$ is the preimage under $F^{e_0}$ of $I^{\fbp{p^{e_0}}}$. So, having explicit upper bounds on $\fte I$ is valuable. 

Unfortunately, the set $\{ \fte I \mid I \subset R\}$ can be unbounded, even in quite nice rings \cite{Bre06}. However, in a variety of contexts, certain classes of ideals admit uniform upper bounds on their Frobenius test exponents. Due to relationships with local cohomology modules we will explicate soon, parameter ideals form a natural and important example of a class which enjoys uniformly upper bounded Frobenius test exponents across a variety of different singularity types for rings. Thus, if $(R,\fm)$ is a local ring, we set $$\fte R := \sup \{ \fte \fq \mid \fq \textrm{ is a full parameter ideal} \},$$ which is called the {\bf Frobenius test exponent of $R$}. 

\begin{remark}\label{rmk: full fte bounds partial fte}
By a similar technique to \cite[Lem. 3.1]{Ma15}, $\fte R$ also bounds the Frobenius test exponent for ideals generated by partial systems of parameters. In particular, consider $(R,\fm)$ a local ring of dimension $\dim R = d$ and $\underline{x}=x_1,\ldots,x_t$ be a part of a system of parameters. Set $e := \fte R$ and assume this is finite. Extend $\underline{x}$ to a full system of parameters $x_1,\ldots,x_t,y_{t+1},\ldots,y_d$. For $n \in \NN$, set $\underline{y}_n = x_1,\ldots,x_t,y_{t+1}^n,\ldots,y_d^n$. Notably, $(\underline{x})^F\subset (\underline{y}_n)^F$ for each $n$. For $z \in (\underline{x})^F$, we have $z^{p^e} \in (\underline{y}_n)^{\fbp{p^e}}$ for all $n$, hence \[z^{p^e} \in \bigcap_n \left(x_1^{p^e},\ldots,x_t^{p^e},\left(y_{t+1}^{p^e}\right)^n,\ldots,\left(y_d^{p^e}\right)^n\right) = (\underline{x})^{\fbp{p^e}},\] where the last equality follows from Krull's Intersection Theorem. Hence we have the desired estimate $\fte (\underline{x}) \le \fte R$.
\end{remark}

\subsection{The Frobenius action on local cohomology}\label{subsec: frob action} A primary approach to address bounds on $\fte R$ relies on local cohomology, a subject with which we assume the reader is familiar. For a local ring $(R,\fm)$ of dimension $d$, the Frobenius map $F:R\rightarrow R$ induces a map $F : H^j_\fm(R) \rightarrow H^j_\fm(R)$, which is called the Frobenius action on local cohomology. Note that $F$ is not $R$-linear but $p$-linear, i.e., for all $r \in R$ and $\xi \in H^j_\fm(R)$, $F(r\xi)=r^pF(\xi)$. We are interested in the elements which vanish under this action, so consider \[0^F_{H^j_\fm(R)} := \{ \eta \in H^j_\fm(R) \left| F^e(\eta) = 0 \textrm{ for some } e \right. \},\] which we call the \textbf{Frobenius orbit closure of $0$ in $H^j_\fm(R)$}.\footnote{There is a notion of Frobenius closure for submodules that is generally distinct from this one, however they agree in the case of the top local cohomology module $H^d_\fm(R)$ for a local ring $(R,\fm)$. We will not need the general notion here, so we focus on the orbit closure.} 

We will also need to address the \textbf{tight closure of $0$ in $H^d_\fm(R)$}, denoted $0^*_{H^d_\fm(R)}$, and is defined\footnote{There is a more involved definition of tight closure of submodules in general, but this works in our context. See \cite[Rmk. 2.6]{PQ19} for further details.} \[
0^*_{H^d_\fm(R)} = \left\lbrace \eta \in H^d_\fm(R) \mid \textrm{there is a } c \in R^\circ \textrm{ such that } cF^e(\eta) = 0 \textrm { for all } e \gg 0 \right\rbrace.
\] From the definition, it is clear $0^F_{H^d_\fm(R)} \subset 0^*_{H^d_\fm(R)}.$ It will be common to simplify this notation to $0^F_j(R)$. When the setting is fixed or clear from context, we may write $0_j^F$ for $0^F_{H^j_\fm(R)}$ or $0^*_d$ for $0^*_{H^d_\fm(R)}$.

\subsection{Hartshorne-Speiser-Lyubeznik numbers}\label{sec:HSL}

As for Frobenius closure, one might hope for a uniform exponent $e_0$ of Frobenius which computes $0^F_{H^j_\fm(R)}$ as $\ker(F^{e_0}:H^j_\fm(R)\rightarrow H^j_\fm(R)$). The existence of such an exponent was first shown in the algebro-geometric setting by Hartshorne and Speiser (\cite[Prop. 1.11]{HS77}) and later by Lyubeznik \cite[Prop. 4.4]{Lyu97} and Sharp \cite{Sha07}, such a $e_0$ exists for any artinian module with a Frobenius action. 

\begin{theorem}\label{thm:HSL}[Hartshorne-Speiser-Lyubeznik] Let $(R,\fm)$ be a local ring. For each $0 \leq j \leq \dim R$, there exists an $e_j \in \mathbb{N}$ such that \[0^F_{H^j_\fm(R)} = \ker \left(F^{e_j}:H^j_\fm(R) \rightarrow H^j_\fm(R)\right).\]  
\end{theorem}

Due to the theorem above, we refer to 
\[\inf \left\{ e \in \mathbb{N} \left| 0^F_{H^j_\fm(R)} = \ker(F^e:H^j_\fm(R)\rightarrow H^j_\fm(R)) \right. \right\} \] as the Hartshorne-Speiser-Lyubeznik (HSL) number of $H^j_\fm(R)$. For notational convenience, we set $\HSL_j(R) := \HSL(H_\fm^j(R))$. Set also $$\HSL(R) := \max \{ \HSL_j(R) \mid 0\le j \le d \}.$$

The connection between $\HSL(R)$ and $\fte R$ was first explicated by Katzman and Sharp, who showed in \cite{KS06} for a Cohen-Macaulay local ring $R$, $\fte R = \HSL(R)$. See \Cref{sec:genfte} for a history of bounding $\fte(R)$ by HSL numbers.

\subsection{\texorpdfstring{$F$}{F}-nilpotent rings and their generalizations}\label{sec:Fnil} Outside of the Cohen-Macaulay setting, results on bounding $\fte(R)$ have typically been attained by replacing the vanishing of the lower local cohomology modules with some degree of nilpotence under Frobenius.  

\begin{definition}
Fix $(R,\fm)$ a local ring of dimension $d$. We say a local cohomology module $H_\fm^j(R)$ is {\bf nilpotent} if $H_\fm^j(R) = 0^F_j$. Moreover, $R$ is said to be {\bf weakly $F$-nilpotent} provided $H_\fm^j(R)$ is nilpotent for all $j < d$, and $R$ is {\bf $F$-nilpotent} if it is weakly $F$-nilpotent and $0^F_{d} = 0^*_{d}$.
\end{definition}

$F$-nilpotent rings are those for which the Frobenius action on local cohomology is as nilpotent as possible, as $0^F_d \subset 0^*_d$ and $H^{\dim R}_\fm(R)$ never vanishes under Frobenius, see e.g. \cite[Lem. 4.2]{Lyu06}.

\

\noindent {\bf Nilpotence of Frobenius actions:} The Frobenius action on local cohomology is an example of a more general theory of Frobenius actions on modules. An $R$-module $M$ together with a $p$-linear map $\rho: M\rightarrow M$, i.e. $\rho$ is additive and $\rho(rm) = r^p\rho(m)$ for all $r \in R$ and $m \in M$ is a module $M$ with {\bf Frobenius action $\rho$}. We write $0^\rho_M$ for $\cup_e \ker(\rho^e:M\rightarrow M)$ and say $M$ is {\bf nilpotent under $\rho$} if there is an $e_0 \in \mathbb{N}$ such that $M = \ker(\rho^{e_0}:M\rightarrow M)$. As for $H^j_\fm(R)$, we call the least $e$ such that $0^\rho_M = \ker(\rho^e:M\rightarrow M)$ the HSL number of $\rho$, which is finite if $M$ is finitely generated, or if $M$ is artinian by \Cref{thm:HSL}.

\begin{remark}\label{rmk: hsl drops when localizing}
    Notice that if $M$ is a finitely generated $R$-module with Frobenius action $\rho$, for any multiplicative set $S\subset R$, we have that $S^{-1}M$ has a natural Frobenius action, which is defined by $\rho_\fp(m/t) = \rho(m)/t^p$. Furthermore, if $e_0$ is the HSL number of $M$, the HSL number of $S^{-1}M$ is no more than $e_0$. To see this, note if $m/t$ is a nilpotent element of $S^{-1}M$, then $\rho^e_\fp(m/t) = \rho^e(m)/t^{p^e} = 0$ so there is an $s \in S$ such that $s\rho^e(m) = 0$ in $M$. Hence, $s^{p^e}\rho^e(m) = \rho^e(sm) = 0$, and thus $\rho^{e_0}(sm) = 0$. We have then $0=\rho^{e_0}_\fp(sm/t) = s^{p^{e_0}}\rho^{e_0}(m/t)$ so $\rho^{e_0}(m/t) =0$.
\end{remark}

Two important examples of modules with Frobenius actions we will discuss later in the survey include $JH^j_\fm(R)$ for $J$ an ideal of a local ring $R$, whose Frobenius action is given by restriction of the Frobenius action $F$ on $H^j_\fm(R)$, and the cokernel of a ring extension $\phi:R\rightarrow S$, where the Frobenius action on $S/\phi(R)$ is given by $\overline{F}(s+\phi(R)) = s^p + \phi(R)$. 

\

\noindent {\bf Relative Frobenius maps:} Another important class of Frobenius-like maps arises as follows. Fix a local ring $(R,\fm)$. For an ideal $I$ in $R$ the iterated Frobenius map $F^e : R/I\rightarrow R/I$ factors via the following commutative diagram. \begin{center}
\begin{tikzcd}
R/I \arrow{rr}{F^e} \arrow[swap]{dr}{f^e} & \, & R/I \\ 
\, & R/I^{\fbp{p^e}} \arrow[swap]{ur}{\pi_e}  
\end{tikzcd}
\end{center} The map $\pi_e$ is the usual projection map and $f^e :  R/I\rightarrow R/I^{\fbp{p^e}}$ is defined by $f^e(r+I) = r^{p^e}+I^{\fbp{p^e}}$, which is called the {\bf relative Frobenius map}. This map is related to Frobenius closure of $I$ since $I^F/I = \bigcup \ker (f_e:R/I\rightarrow R/I^{\fbp{p^e}})$. Just as $F:R\rightarrow R$ induces a map on the local cohomology modules $H^j_\fm(R)$, we also have that the relative Frobenius map induces a $p^e$-linear map on the local cohomology modules of $R/I$, which we call the \textbf{relative Frobenius action on local cohomology}. \[
\rho^e : H^j_\fm\left(R/I\right)\rightarrow H^j_\fm\left(R/I^{\fbp{p^e}}\right)
\] Furthermore, by applying $H^j_\fm(R)$ to the diagram above, we also factor the Frobenius action $F$ on $H^j_\fm(R)$ via the relative Frobenius action.  
\begin{center}
\begin{tikzcd}
H_\fm^j(R/I) \arrow{rr}{F^e} \arrow[swap]{dr}{\rho^e} & \, & H_\fm^j\left(R/I\right) \\ 
\, & H_\fm^j\left(R/I^{\fbp{p^e}}\right) \arrow[swap]{ur}{\pi_e}  
\end{tikzcd}
\end{center} 

As before, we consider $0^\rho_{H^j_\fm(R/I)} = \cup_e\ker(\rho^e)$. We follow similar notational simplifiations $0^\rho_j(R/I)$ or just $0^\rho_j$ if the setting is clear. If there is an $e$ such that $0^\rho_{H^j_\fm(R/I)} = \ker(\rho^e)$, we say the least such $e$ is the \textbf{relative HSL number of $H^j_\fm(R/I)$}, which we write $\HSL^j_\rho(R/I)$.

Finally, modules with Frobenius actions, nilpotence, and other considerations in this section have natural analogues in the graded setting. We do not consider the graded case in this survey, but discuss it carefully in \cite{MM24}. 

\begin{example} We review a number of examples of weakly $F$-nilpotent and $F$-nilpotent rings. For simplicity, in the list that follows, $k$ will denote a perfect field of prime characteristic $p>2$.
\begin{itemize}
\item Any Cohen-Macaulay ring is weakly $F$-nilpotent, and any $F$-rational ring is $F$-nilpotent.
\item The hypersurface $k[[x,y,z]]/(x^4 + y^4 - z^4)$ is $F$-nilpotent when $p \equiv 3 \bmod 4$ as shown by Blickle in \cite[Ex. 5.28]{Bli04}. More generally, each hypersurface $k[[x_0,\ldots,x_n]]/(x_0^d + \ldots + x_{n-1}^d - x_n^d)$ with $p \equiv -1 \bmod d$, $n \geq 2$ is $F$-nilpotent \cite[Lem. 4.19]{MM24}. 
\item Direct summands of (weakly) $F$-nilpotent rings are (weakly) $F$-nilpotent. Thus, any Veronese subring of a weakly $F$-nilpotent graded ring remains weakly $F$-nilpotent. 
\item The ring $k[x^4,x^3y,xy^3,y^4]$ is $F$-nilpotent, as are most \textit{pinched Veronese} rings (see \cite[Thm. B]{MP23}) formed by removing an algebra generator from the $d$th Veronese subring of a polynomial ring. This example shows that $F$-nilpotent rings are not necessarily normal.
\item Every numerical semigroup ring $k[x^{n_1},\ldots,x^{n_t}]$ is $F$-nilpotent by \cite[Ex. 3.8]{MS25} and \cite[Cor. 4.4]{MS25}. 
\item The hypersurface $R=k[[x,y,z]]/(x^2+y^3+z^7+xyz)$ is normal and not $F$-nilpotent \cite[Rmk. 3.15]{DMP24} and \cite[Ex. 2.7]{ST17}, though the cusp $k[[x,y]]/(x^2+y^3)=R/zR$ is $F$-nilpotent. Thus, as noted by Srinivas-Takagi, $F$-nilpotence does not deform (see the following remark) since $z \in R$ is a regular element. 
\end{itemize}
\end{example}

\begin{remark}
A property $\mathcal{P}$ of local rings is said to \textit{deform} if, whenever $R$ is a local ring such that $R/zR$ has the property $\mathcal{P}$ for some regular element $z$ in $R$, then $R$ also has the property $\mathcal{P}$. The deformation question for $F$-singularities is an interesting and storied problem. It is well-known, \cite{HH94}, that $F$-rationality deforms, and whether $F$-injectivity deforms remains  difficult and still largely open question (\cite{HMS14},\cite{DSM22}). Weak $F$-nilpotence deforms in certain contexts (see \cite[Sec. 5]{CMM23} for some sufficient conditions), and although $F$-nilpotence does not deform as noted above, \cite[Thm. B]{PQ19} connects $F$-nilpotence of $R$ with a relative notion of $F$-nilpotence for $R/xR$ as $x$ ranges over all \textit{filter regular} elements of $R$. Filter regular sequences are reviewed in Section~\ref{sec:genfte}.
\end{remark} Under our standard assumptions (e.g. $F$-finite), \cite[Prop. 2.8]{PQ19} guarantees that the local ring $(R,\fm)$ is $F$-nilpotent if and only if the reduction $R/\sqrt{0}$ is $F$-nilpotent and if and only if the $\fm$-adic completion of $R$ is $F$-nilpotent. 

\

\noindent {\bf Generalized $F$-nilpotence:} In the sense that nilpotence under Frobenius is similar to vanishing, we can view weakly $F$-nilpotent rings as similar to Cohen-Macaulay rings, where vanishing of local cohomology is replaced with nilpotence of the local cohomology under $F$. Another, more classical, route to generalize Cohen-Macaulayness is by replacing vanishing with finite length. Recall a local ring $(R,\fm)$ is \textit{generalized Cohen-Macaulay} if the local cohomology modules $H^j_\fm(R)$ are finite length for all $0 \le j < \dim R$. When $R$ has a dualizing complex (e.g. $R$ is $F$-finite), generalized Cohen-Macaulayness is equivalent to $R$ being Cohen-Macaulay on the punctured spectrum. Further, generalized Cohen-Macaulay rings were shown to have finite Frobenius test exponent in \cite{HKSY06} shortly after Katzman-Sharp's proof for Cohen-Macaulay rings. 

In \cite[Main Thm]{Quy19}, Quy proved weakly $F$-nilpotent local rings have a finite Frobenius text exponent, and he showed a similar proof also works for   generalized Cohen-Macaulay rings. Extending these results, in \cite{Mad19}, the first author combined the generalized Cohen-Macaulay and weakly $F$-nilpotent hypotheses.

\begin{definition} A local ring $(R,\fm)$ is {\bf generalized weakly $F$-nilpotent} provided each $H_\fm^i(R)/0_i^F$ is finite length for all $i < \dim R$. 
\end{definition}

Just as Lyubeznik observed that the top local cohomology module $H^{\dim R}_{\fm}(R)$ of a local ring is not nilpotent under the Frobenius action, the top dimensional quotient $H^{\dim R}_\fm(R)/0^F_{\dim R}$ cannot be finite length, see \cite[Lem. 3.4.(c)]{MM24}. The main result of \cite{Mad19} shows that generalized weakly $F$-nilpotent local rings also enjoy finite Frobenius test exponents, which we discuss further in \Cref{sec:genfte}. For excellent local rings, generalized weakly $F$-nilpotent is equivalent to weakly $F$-nilpotent on the punctured spectrum, see \cite[Prop. 4.6]{KMPS23}.

\begin{example}  We review some examples of generalized weakly $F$-nilpotent rings. For simplicity, in the list that follows, $k$ will denote a perfect field of prime characteristic $p>2$.
\begin{itemize}
\item Any generalized Cohen-Macaulay or weakly $F$-nilpotent local ring is generalized weakly $F$-nilpotent, just as Cohen-Macualay rings are weakly $F$-nilpotent. \begin{center}
\begin{tikzpicture}
\node[scale=1] at (0,0) {
\begin{tikzcd}
\text{Cohen-Macaulay} \arrow[Rightarrow]{r} \arrow[Rightarrow]{d} & \text{gen. Cohen-Macaulay}  \arrow[Rightarrow]{d}  \\
\text{weakly {\em F}-nilpotent} \arrow[Rightarrow]{r} & \text{gen. weakly {\em F}-nilpotent}
\end{tikzcd}};
\end{tikzpicture}
\end{center} 
\item  There is a way \cite[Ex. 4.8]{KMPS23} to construct rings $R$ which are generalized weakly $F$-nilpotent but are neither weakly $F$-nilpotent nor generalized Cohen-Macaulay. 
\item The Segre product $k[x,y,z]/(x^4+y^4-z^4) \# k[u,v]$, which is the coordinate ring of a Segre product of $\mathbf{P}^1$ with a Fermat curve, is shown in \cite[Ex. 5.6]{MM24} to be generalized weakly $F$-nilpotent in general, and weakly $F$-nilpotent when $p \equiv 3 \bmod 4$. 
\end{itemize}
\end{example}

\section{\texorpdfstring{$F$}{F}-depth, generalizations, and numerical aspects}

In this section, we discuss notions of depth-like invariants which measure nilpotence under Frobenius. The first we consider, the $F$-depth of a local ring, was introduced by Lyubeznik in \cite{Lyu06} and is likely more familiar to the audience. This notion helps to explicate the analogy that weakly $F$-nilpotent singularities are to be considered as generalizations of Cohen-Macaulay singularities. Inspired by the generalized depth of Huckaba and Marley introduced in \cite{HM94}, the second depth-like invariant we consider is the generalized $F$-depth with respect to an ideal introduced in \cite{CMM23}. Both of these notions of generalized depth incorporate an auxiliary ideal which controls the locus of failure of Cohen-Macaulayness or weak $F$-nilpotence respectively.

\subsection{\texorpdfstring{$F$}{F}-depth} We start with Lyubeznik's $F$-depth and its basic properties. Lyubeznik introduced $F$-depth to study a problem of Grothendeick about \textit{cohomological dimension} of ideals.

\begin{definition}
Let $(R,\fm)$ be a local ring of prime characteristic $p>0$. The \textbf{$F$-depth of $R$}, denoted $\fdp R$, is the smallest index $j$ for which $H^j_\fm(R)$ is not nilpotent, i.e., the first index $j$ for which $0^F_j \neq H^j_\fm(R)$.
\end{definition}

As noted before, Lyubeznik showed that $\fdp R \le \dim R$ in \cite{Lyu06} for a local ring $(R,\fm)$. Furthermore, $\fdp R = \dim R$ if and only if $R$ is weakly $F$-nilpotent. We also have $\fdp R > 0$ if $\dim R>0$ since $H^0_\fm(R)$ is inside $\sqrt{0}$ which is nilpotent under Frobenius. Hence, if $R$ is a local ring of positive dimension, $0 < \fdp R \le \dim R$. 

The main result of \cite{Lyu06} states that if $(R,\fm)$ is a regular local ring of prime characteristic $p>0$ and dimension $d$, and $I$ is an ideal in $R$, then $H^{d-i}_I(R) = 0$ if and only if $H^i_\fm(R/I)$ is nilpotent under Frobenius. Correspondingly, the cohomological dimension of $I$ is at most $r$ if and only if $\fdp(R/I) \ge d-r$.

\subsection{Generalized \texorpdfstring{$F$}{F}-depth}\label{sec:genfdepth} Recall, from the introduction, one of our aims is to review the literature on nilpotent singularities through the lens of depth-like invariants. Our central focus in this subsection is a generalized depth-like invariant introduced in \cite{CMM23}. Owing to space limitations, we refer to the reader to article for proofs of the results.

\begin{definition}
Fix $(R,\fm)$ a local ring and $J \subset R$ an ideal. The {\bf generalized $F$-depth with respect to $J$} is 
\[\gfdp_J(R) := \inf \left\lbrace j \in \NN \left| J^N H_\fm^j(R) \not \subset 0^F_j \text{ for any } N \in \NN \right\rbrace \in \NN \cup \{ \infty \right. \}.\]
\end{definition}

\begin{remark} 
Notice that when $J=R$, $\gfdp_R(R) = \fdp R$, so weak $F$-nilpotence can be studied under the lens of $\gfdp_J$. Further, $\gfdp_\fm(R) = t$ if and only if $H_\fm^i(R)/0_i^F$ is finite length for all $i < t$. 

The invariant $\gfdp_\fm(R)$ was introduced in \cite{MM24} as {\bf generalized $F$-depth} $\gfdp R$. It is direct from the definition that $R$ is generalized weakly $F$-nilpotent if and only if $\gfdp(R) = \dim R$. 
\end{remark}

More classically, one also can consider the finiteness dimension of a local ring, which measures the first index for which $H^j_\fm(R)$ is not finite length, or the generalized depth with respect to an ideal $J$ introduced by Huckaba and Marley in \cite{HM94}, which measures the first index for which $H^j_\fm(R)$ is not annihilated by any power of $J$. Notably, the finiteness dimension is the generalized depth with respect to $\fm$. 

We enumerate some basic facts about $\gfdp_J(R)$ for an arbitrary ideal $J$ of $R$, which then also hold for $F$-depth and generalized $F$-depth; for proofs, see \cite[Thm. 3.8, Thm. 3.10]{CMM23}.
\begin{enumerate}
\item The quantity $\gfdp_J(R)$ is invariant under flat local extensions, thus, in particular we have $\gfdp_J(R) = \gfdp_{J\widehat{R}}(\widehat{R})$, where $\widehat{R}$ is the $\fm$-adic completion of $R$.
\item For $R_{\textrm{red}} := R/\sqrt{0}$, $\gfdp_J(R) = \gfdp_{JR_\textrm{red}}(R_{\textrm{red}})$. Note the analogous statement fails for depth. 
\item If $R$ is $F$-finite and equidimensional, $\gfdp_J(R) > \dim R$ if and only if $J \subset \sqrt{0}$. In particular, if $J$ is a nonzero ideal of a reduced ring, $\gfdp_J(R) \le \dim R$.
\end{enumerate}

To study $\gfdp_J(R)$ as an invariant of local rings, it is convenient to label certain cohomological annihilators as follows. Set $\fb_i(R) := \sqrt{ \Ann_R\left(H_\fm^i(R)/0^F_i\right) }$. From the definitions, it immediately follows that $\gfdp_J(R) = k$ if and only if $J \subset \fb_0(R) \cap \cdots \cap \fb_{k-1}(R)$ and $J \not\subset \fb_k(R)$.

\begin{lemma}\label{lem:singlej} For local ring $(R,\fm)$ and ideal $J \subset R$, the following are equivalent:
\begin{enumerate}
\item $J H^i_\fm(R)$ is nilpotent under the Frobenius action on $H^j_\fm(R)$,	
\item $J^N H^i_\fm(R)$ is similarly nilpotent for some $N$,
\item for each $x \in J$ and all $e \gg 0$, $x^{p^e}F^e : H^i_\fm(R) \to H^i_\fm(R)$ is the zero map,
\item $J \subset \fb_j(R)$. 
\end{enumerate}
\end{lemma}
\begin{proof}
Clearly (1) implies (2). For proof the equivalence the implications (2) implies (3) implies (4), see \cite[Lem. 3.2, Lem. 3.5]{CMM23}. Finally, (4) implies (1) as for any $x \in J$ and $\eta \in H^i_\fm(R)$, $F^e(x \eta) = x^{p^e}F^e(\eta) \in 0^F_i$ for $e \gg 0$ sufficient to ensure $x^{p^e} \in \Ann_R\left(H_\fm^i(R)/0^F_i\right)$.
\end{proof}

The following theorem is a reformulation of \cite[Prop. 4.6]{KMPS23}, which gives an equivalence of generalized weak $F$-nilpotence with $F$-nilpotence on the punctured spectrum. We improve the result by replacing the punctured spectrum with $V(J)$.

\begin{theorem}\label{thm:replacepuncspec}
Let $(R,\fm)$ be an $F$-finite local ring of prime characteristic $p>0$ which is equidimensional. For any ideal $J\subset R$, we have $\gfdp_J(R) \geq t$ if and only if $\fdp(R_\fp) \geq t-d+\hgt \fp$ for all $\fp \not\in V(J)$. 
\end{theorem}
\begin{proof}
Without loss of generality, we may assume $R$ is complete by \cite[Thm. 3.8]{CMM23}. Now, if $\gfdp_J(R)\ge t$, then $J \subset \fb_j(R)$ for all $0 \le j < t$. Consequently, for $\fp \not \in V(J)$, $\fb_j(R)R_\fp = R_\fp$ for $0 \le j < t$. However, by \cite[Lem. 3.9]{CMM23}, this means $\fb_{j-d+\hgt \fp}(R_\fp) = R_\fp$ for all $0 \le j < t$, in particular, $\fdp(R_\fp) \ge t-d+\hgt \fp$.

Conversely, suppose $\fdp(R_\fp) \ge t-d+\hgt\fp$ for all $\fp \not \in V(J)$. Thus, $V(\fb_j(R)) \subset V(J)$ for all $0 \le j < t$ by \cite[Lem. 3.9]{CMM23} again. Hence $J \subset \fb_j(R)$ for $0 \le j < t$, and thus $\gfdp_J(R) \ge t$ by \Cref{lem:singlej}.
\end{proof}

\section{Geometric properties of \texorpdfstring{$F$}{F}-nilpotent rings and their generalizations}

In this section, we consider geometric properties of $F$-nilpotent rings and their generalizations. In particular, it is common to ask whether an $F$-singularity type determines a Zariski open subset of $\Spec R$. In the setting of $F$-finite rings, this is well-known for $F$-rational (due to Velez \cite{Vel95}) and $F$-injective (due to Schwede \cite{Sch09}) singularities, so it is natural to ask this question for $F$-nilpotent type singularities as well. This question was answered in \cite[Thm. 5.2]{KMPS23}, under the assumption that $R$ is $F$-finite or essentially of finite type over an excellent local ring. This proof adapts to the more general setting of the locus of primes for which $\gfdp_J$ is bounded below.

\begin{theorem}\label{thm:gfdpopen}
Let $(R,\fm)$ be an $F$-finite, equidimensional local ring, $J\subset R$ an ideal, and $t \ge 0$. The locus \[ \left\lbrace \fp \in \Spec R \left| \gfdp_{JR_\fp}(R_\fp) \ge t\right.\right\rbrace\] is open.
\end{theorem}
\begin{proof}
Write $d =\dim R$. We can reduce to the case that $R$ is complete and equidimensional, since $\gfdp_J(R)$ is invariant under completion. We can then apply \cite[Lem. 3.9]{CMM23} and \Cref{lem:singlej} to see that $(JR_\fp)H^j_{\fp R_\fp}(R_\fp)$ is nilpotent if and only if $JR_\fp$ is in $\fb_{j}(R_\fp) = \fb_{j+d-\hgt \fp}(R) R_\fp$. Hence, if $\gfdp_{JR_\fp}(R_\fp)\ge k$, then $JR_\fp H^j_{\fp R_\fp}(R_\fp)$ is nilpotent for $0 \le j < k$, so we need $JR_\fp$ to be inside $$\fb_0(R_\fp) \cap \cdots \cap \fb_{k-1}(R_\fp) = (\fb_{d-\hgt\fp}(R) \cap\cdots \cap \fb_{k-1+d-\hgt\fp}(R))R_\fp.$$ From here, it is easy to see that the locus is open, as it is the set of primes which are not in the support of the ideal $J+\fb_{d-\hgt\fp} (R) \cap \cdots \cap \fb_{k-1+d-\hgt\fp}(R)$ in $R/\left(\fb_{d-\hgt\fp} (R) \cap \cdots \cap \fb_{k-1+d-\hgt\fp}(R)\right)$.
\end{proof}

In addition to local cohomology methods, an important way to bound Frobenius test exponents is via finite, purely inseparable extensions. Recall a ring extension (i.e. an injective map of rings) $\phi : R \rightarrow S$ of rings of prime characteristic $p>0$ is \textit{purely inseparable} if, for each $s \in S$, we have $s^{p^e} \in R$ for some $e \in \NN$. This is equivalent to the Frobenius action on the cokernel $C=S/R$ given by $s+R \mapsto s^p+R$ being nilpotent. 

Purely inseparable extensions are interesting geometrically, as when $\phi : R \rightarrow S$ is purely inseparable, the induced map on spectra is a homeomorphism. Furthermore, Frobenius properties of $S$ tend to descend to $R$ ``up to nilpotence," a perspective explored and utilized in \cite{DMP24}, \cite{MP23}, and \cite{MS25}. 

Fix $\phi \colon (R,\fm_R) \rightarrow (S,\fm_S)$ a finite, purely inseparable ring extension of local rings such that $S^{p^{e_0}} \subset R$. In this context, we have $\sqrt{\fm_R S} = \fm_S$, so we can compute the local cohomology modules $H^j_{\fm_S}(S)$ as $H^j_{\fm_R}(S)$ instead. The following theorem is a restatement and extension of \cite[Thm. 2.17]{DMP24}.

\begin{theorem} 
Let $R\rightarrow S$ be a finite, purely inseparable map of local rings. For $J\subset R$ an ideal, $\gfdp_J(R) \geq t$ if and only if $\gfdp_{JS}(S) \geq t$. In particular, $R$ is (generalized) weakly $F$-nilpotent if and only if $S$ is. Similarly, $R$ is $F$-nilpotent if and only if $S$ is.
\end{theorem}

\begin{proof}
From the short exact sequence $0 \to R\to S\to C \to 0$, we obtain a long exact sequence in local cohomology as below. \[ \ldots \to H^{j-1}_{\fm_R}(C) \to H^j_{\fm_R}(R) \to H^j_{\fm_R}(S) \to H^j_{\fm_R}(C) \to H^{j+1}_{\fm_R}(R) \to \ldots\] Since the Frobenius action on $C$ is nilpotent, so is the Frobenius action on $H^j_{\fm(R)}(C)$ for all $j$, and consequently $JH^j_{\fm_R}(R)$ is nilpotent if and only if \[JH^j_{\fm_R}(S)=(JS)H^j_{\fm_S}(S)\] is nilpotent by \cite[Thm. 3.5]{MM24}, from which the first and second claim follow. The final claim follows from \cite[Thm. 2.17]{DMP24}. 
\end{proof}

In the context of Frobenius test exponents, we can exploit $e_1=\fte(IS)$ to control $\fte I$. In particular, if $x \in I^F$, then $x \in I^FS \subset (IS)^F$. So one has $x^{p^{e_1}} \in (IS)^{\fbp{p^{e_1}}}$ and can then conclude that $x^{p^{e_0+e_1}} \in I^{\fbp{p^{e_0+e_1}}}R$, so $\fte I \le e_0+e_1$. In particular, if $S$ has a uniform bound on the Frobenius test exponent of a class of ideals, then a similar bound ought to extend to the same class of ideals in $R$. The most extreme case is where the Frobenius test exponent of any ideal of $S$ is $0$, which is equivalent to $F$-purity in this case. Compare the following with \cite[Thm. 4.4 and Rmk. 4.5]{DMP24}.

\begin{theorem}\label{thm:FTEasscent}
For $R \to S$ a finite, purely inseparable local extension of local rings, $\fte R$ is finite if and only if $\fte S$ is finite. In particular, if $S$ is $F$-pure, then $\fte I \le e_0$ for any ideal $I\subset R$, where $e_0$ is least such that $F^{e_0}(S)\subset R$.
\end{theorem}

Given a reduced, $F$-finite ring $R$, there is a canonical finite, purely inseparable extension one may consider, i.e., the weak normalization $R^{\WN}$. %This is also written $^*R$ in the literature. 

\begin{center}
\begin{tikzcd}
    R \arrow{dr}\arrow{rr} & \, & R^{\textrm{N}} \\
    \, & R^{\WN} \arrow{ur} & \,
\end{tikzcd}
\end{center}

The weak normalization is the collection of elements $x$ of the total ring of quotients of $R$ which satisfy an equation of the form $x^{p^e}=r$ for some $e \in \NN$ and $r \in R$, that is, the weak normalization is the set of $p^e$-th roots of elements of $R$ living in the total ring of quotients of $R$. Clearly $R^{\WN}$ is inside the normalization of $R$ and as we are assuming $F$-finite, $R\rightarrow R^{\textrm{N}}$ and consequently $R\rightarrow R^{\WN}$ are finite extension of $R$. Further, $R\rightarrow R^{\WN}$ is purely inseparable by definition. Hence, by the previous theorem, a local reduced $F$-finite ring $R$ has finite Frobenius test exponent if and only if $R^{\WN}$ has finite Frobenius test exponent (cf. \cite[Cor. 4.2]{DMP24}.

The weak normalization also plays a role in the study of $F$-nilpotent rings. In \cite{MP23}, the authors studied pinched Veronese rings, which are formed by removing an algebra generator from a Veronese subalgebra of a polynomial ring over a field. To consider when pinched Veronese rings were $F$-nilpotent, it was determined that the behavior was governed by the map from the pinched Veronese to the full Veronese, which is $F$-pure (and more). For these examples, the normalization map is purely inseparable, and in \cite[Thm. 1.2]{DMP24}, this was shown to be a necessary condition for $F$-nilpotence in general.

\begin{theorem}
Let $(R,\fm)$ be a reduced local ring. If $R$ is $F$-nilpotent, $R \rightarrow R^{\textrm{N}}$ is purely inseparable (equivalently, $R^{\WN}=R^{\textrm{N}}$), and the converse holds either when $\dim R = 1$ or if $R$ is an affine semigroup ring.
\end{theorem}

For rings of dimension 1, i.e. curves, pure inseparability of the normalization map is also related to the number of \textit{geometric branches}, that is, the number of minimal primes of the strict henselization of $R$. In particular, for excellent reduced curves, $R$ is geometrically unibranched (i.e. the number of geometric branches is $1$) at each closed point if and only if $R$ is $F$-nilpotent, \cite[Thm. 1.2]{DMP24}. Furthermore, the authors used tight and Frobenius closure to count the number of geometric branches $b(\fm)$ at a closed point $\fm$ of $\Spec R$ as \[
b(\fm) = \dim_{R / \fm} \left(0^*_{H^1_{\fm R_\fm}(R_\fm)}/0^F_{H^1_{\fm R_\fm}(R_\fm)}\right) + 1 = \dim_{R/\fm}\left((x)^*/(x)^F\right) + 1
\] where $x$ is a sufficiently general parameter element of $R_\fm$.

\begin{remark}\label{rmk: pi locus}
For a general finite ring extension $\phi: R\rightarrow S$ which is not necessarily purely inseparable, we can consider the locus \[X_\phi = \{ \fp \in \Spec R \left| \phi \otimes_R R_\fp \text{ is purely inseparable} \right.  \}.\] On $X_\phi$, $\phi_\fp: R_\fp \rightarrow S\otimes_R R_\fp$ is purely inseparable, and so $S\otimes_R R_\fp = S_\fq$ is another local ring as purely inseparable maps are homeomorphisms on spectra. In particular, for each $\fp \in X_\phi$ there is a unique $\fq \in \Spec S$ such that $\fq$ lies over $\fp S$, and $\phi_\fp :R_\fp \rightarrow S_\fq$ is a purely inseparable local homomorphism of local rings.

Furthermore, the set $X_\phi$ is open in $\Spec R$ since it is the complement of the support of the finitely generated $R$-module; $S/R$ modulo its nilpotent submodule $0^{\overline{F}}_{S/R}$ under the action $\overline{F}(s+R) = s^p+R$.

\end{remark}
\begin{theorem}\label{FTE:open}
Suppose $\phi \colon R\rightarrow S$ is a finite ring extension of $F$-finite rings of prime characteristic $p>0$. Adopting the notation of \Cref{rmk: pi locus}, there is an $e_0 \in \NN$ such that for all $\fp$ in $X_\phi$ such that $S_\fq$ is $F$-pure, $\fte I \le e_0$ for all ideals $I\subset R_\fp$. In particular, if the image of the $F$-pure locus of $S$ is also open in $\Spec R$, then there is an open set in $\Spec R$ on which the Frobenius test exponent of all ideals has a uniform upper bound.
\end{theorem}

\begin{proof} 
Let $e_0$ be the HSL number of $S/R$. By \Cref{rmk: hsl drops when localizing}, we have that for all $\fp \in X_\phi$, the HSL number of $S_{\fq}/R_\fp$ is at most $e_0$. But as $S_\fq$ is $F$-pure and $\phi_\fp:R_\fp\rightarrow S_\fq$ is finite and purely inseparable, $\fte I \le e_0$ for any ideal $I\subset R_\fp$ by \cite[Rmk. 4.5]{DMP24}.
\end{proof}

\section{Applications finiteness of Frobenius test exponents of parameter ideals}\label{sec:genfte}  
To begin this section, we review properties of filter regular sequences, which stand in for regular sequences in rings which are not Cohen-Macaulay.

\begin{definition}
    Let $(R,\fm)$ be a local ring. An element $x \in R$ is a \textbf{filter regular element} if the ideal $(0:_R x)$ is a finite length $R$-module, equivalently if $x \not \in \fp$ for any $\fp \in \operatorname{Ass}(R) \setminus \{\fm\}$. A sequence $x_1,\ldots,x_t$ is a \textbf{filter regular sequence} if the image of $x_i$ in $R/(x_1,\ldots,x_{i-1})$ is a filter regular element of $R/(x_1,\ldots,x_{i-1})$.
\end{definition}

\begin{remark}\label{rmk: filter regular facts}
If $\fq$ is a full parameter ideal in a local ring $(R,\fm)$ then $\fq$ can be generated by a filter regular sequence of length equal to $\dim R$. Furthermore, if $x_1,\ldots,x_t$ is a filter regular sequence, so is $x_1^{n_1},\ldots,x_t^{n_t}$ for each $n_i \in \NN$. Finally, if $x_1,\ldots,x_t$ is a filter regular sequence, then $(x_1,\ldots,x_{i-1}):_R x_i$ is a finite length $R$-module for each $1 \le i \le t$, in particular, the support of $(x_1,\ldots,x_{i-1}):_R x_i$ is contained in $\{\fm\}$ for each $i$. Thus, for each $\fp \neq \fm$, $x_1/1,\ldots,x_t/1$ is a (possibly improper) regular sequence in $R_\fp$. 
\end{remark}

Now fix an ideal $J \subset R$ and a natural number $0 \leq t \leq d$. Consider the set $$\fQ_{J,t} := \{ \fq = (x_1,\ldots,x_t) \subset J \left| x_1,\ldots,x_t \textrm{ is a filter regular sequence} \right. \}.$$ Write $\fte \fQ_{J,t} := \sup\{\fte \fq \mid \fq \in \fQ_{J,t}\}.$ If $R$ is generalized weakly $F$-nilpotent, $\fte \fQ_{J,t}$ is finite for any $J$ and $1 \le t \le d$  by \Cref{rmk: full fte bounds partial fte}. Our aim now is to show how to use generalized $F$-depth statements to bound $\fte \fQ_{J,t}$ in a weaker setting, which generalizes the main result of \cite{HQ22}.  

To give context, we quickly review how nilpotent local cohomology modules and their HSL numbers assist in bounding Frobenius test exponents of parameter ideals following the approaches of \cite{Mad19} and \cite{Quy19}. Note that, for a full parameter ideal $\fq$, $H_\fm^0(R/\fq) = R/\fq$ and $0^\rho_0(R/\fq) = \fq^F/\fq$. From this, one can see $\fte \fq = \HSL^\rho_0(R/\fq)$. This bound depends on $\fq$ -- to produce a uniform  bound on $\fte \fq$, we build an inductive bound on $\HSL_0^\rho(R/\fq)$ via $\HSL^j(R)$. 

Let $t \leq \dim R$, $\fq = (x_1,\ldots,x_t)$ be a parameter ideal generated by a filter regular sequence, and set $\fq_i := (x_1,\ldots,x_i)$ for $i \leq t$. For each $i$ and each $e\in \NN$, we have the commutative diagram of short exact sequences below whose horizontal maps are $R$-linear and vertical maps are $p^e$-linear \begin{center}
    \begin{tikzcd}
        0 \arrow{r} & R/\left(\fq_{i-1}:_R x_i\right) \arrow{r}{\cdot x_i} \arrow{d}{(f^e)'} & R/\fq_{i-1}\arrow{d}{f^e} \arrow{r} & R/\fq_i\arrow{d}{f^e} \arrow{r} & 0 \\
        0 \arrow{r} & R/\left(\fq_{i-1}^{\fbp{p^e}}:_R x_i^{p^e}\right) \arrow{r}{\cdot x_i^{p^e}} & R/\fq_{i-1}^{\fbp{p^e}} \arrow{r} & R/\fq_i^{\fbp{p^e}} \arrow{r} & 0
    \end{tikzcd}
\end{center} where $f^e$ is the iterated relative Frobenius map and $(f^e)'$ is the composition \[
R/\left(\fq_{i-1}:_R x_i\right) \xrightarrow{f^e} R/(\fq_{i-1}:_R x_i)^{\fbp{p^e}} \xrightarrow{\pi} R/\left( \fq_{i-1}^{\fbp{p^e}}:_R x_i^{p^e}\right).
\] Since $x_1,\ldots,x_t$ is a filter regular sequence, the map $(f^e)'$ induces on local cohomology agrees with the relative Frobenius map $\rho^e$ on the local cohomology modules and for $j>0$ we have $H^j_\fm(R/(\fq_{i-1}:_R x_i)) \simeq H^j_\fm(R/\fq_{i-1})$.

Applying $H^j_\fm$ to the commutative diagram above, we have the following commutative diagram whose rows are long exact sequences of $R$-linear maps and whose vertical maps are $p^e$-linear. \begin{center}
    \begin{tikzcd}\label{fig:ladder}
        \cdots \arrow{r}{\cdot x_i} & H^j_\fm(R/\fq_{i-1}) \arrow{d}{\rho^e}\arrow{r}{\alpha_0} & H^j_\fm(R/\fq_i) \arrow{r}{\delta_0}\arrow{d}{\rho^e} & H^{j+1}_\fm(R/\fq_{i-1}) \arrow{r} \arrow{d}{\rho^e}& \cdots \\
        \cdots \arrow{r}{\cdot x_i^{p^e}} & H^j_\fm\left(R/\fq_{i-1}^{\fbp{p^e}}\right) \arrow{r}{\alpha_e} & H^j_\fm\left(R/\fq_i^{\fbp{p^e}}\right) \arrow{r}{\delta_e} & H^{j+1}_\fm\left(R/\fq_{i-1}^{\fbp{p^e}}\right) \arrow{r} & \cdots
    \end{tikzcd}
\end{center} From this ladder, we can chase the diagram to inductively obtain bounds \[\HSL_j^\rho(R/\fq_i) \le \sum_{k=j}^{i+j} \binom{i}{k-j} \HSL_k(R)\] so long as the preimage of $0^\rho_j\left(R/\fq_i^{\fbp{p^e}}\right)$ under $\alpha_e$ is nilpotent under $\rho$; see the proof of \cite[Thm. 3.1]{Mad19}. It is natural to expect depth-like conditions will ensure this preimage is nilpotent. \begin{enumerate}
    \item If $R$ is Cohen-Macaulay (i.e. $\operatorname{depth} R = d$), then $H^j_\fm\left(R/\fq_{i-1}^{\fbp{p^e}}\right)$ vanishes, and so the preimage described above is nilpotent. Katzman-Sharp showed $\fte R = \HSL R$ in this case in \cite{KS06}.
    \item If $R$ is generalized Cohen-Macaulay (i.e. $\findim R = d$), then one may find $e \gg 0$, uniform in $\fq$, so that $x_i^{p^e} H^j_\fm\left(R/\fq_{i-1}^{\fbp{p^e}}\right)$ vanishes. Hence $\alpha_e$ is injective, so the preimage is nilpotent. Huneke-Katzman-Sharp-Yao showed in \cite{HKSY06} $\fte R <\infty$ in this case by different methods, but Quy showed in \cite{Quy19} an explicit bound of the form $\fte R \le e_0 + \sum_{k=0}^d \binom{d}{k} \HSL_k(R)$. 
    \item If $R$ is weakly $F$-nilpotent (i.e. $\fdp R = d$), then by exploiting the fact that $H^j_\fm(R)$ is nilpotent for all $j<d$, Quy obtained a similar bound in \cite{Quy19} of the form $\fte R \le \sum_{k=0}^d \binom{d}{k}\HSL_k(R)$. 
    \item If $R$ is generalized weakly $F$-nilpotent (i.e. $\gfdp R = d$), in \cite{Mad19}, the first named author showed that the previous two proofs could be combined to show that there is an $e_0$ such that $\fte R \le e_0 + \sum_{k=0}^d \binom{d}{k} \HSL_k(R)$.
    \item If $\gfdp R \ge t$, then Huong-Quy (\cite{HQ22}) showed that any filter regular sequence of $x_1,\ldots,x_t$ has finite Frobenius test exponent.
\end{enumerate} 

Considering point (5) above, it is natural to ask if any uniformity on the Frobenius test exponents of some class of ideals is attained if $\gfdp_J R \ge t$. It turns out, as long as the filter regular sequence is contained in $J$, then the proof of \cite[Thm. 3.6]{HQ22} still applies.

\begin{theorem}\label{thm: gfdp_J and filter regular seqeunces}
Let $(R,\fm)$ be a local ring of dimension $d$ and of prime characteristic $p>0$, and let $J$ be an ideal of $R$ such that $\gfdp_J(R) \ge t$. There is an integer $C$ so that $\fte \fQ_{J,t}\le C$.    
\end{theorem}

\begin{proof}
Let $x_1,\ldots,x_t \in J$ be a filter regular sequence, and write $\fq = (x_1,\ldots,x_t)$ and $\fq_i = (x_1,\ldots,x_i)$ for $1 \le i \le t$. We remind the reader that for a local ring $(S,\mathfrak{n})$, set $0^\rho_j(S)$ for $0^\rho_{H^j_{\mathfrak{n}}(S)}$ is the relative orbit closure. 

Since $\gfdp_J(R) \ge t$, by an immediate adaptation of \cite[Lem. 3.5]{Mad19}, there is an $n_0$ such that $J^{2^in_0}H^j_\fm(R/\fq_i) \subset 0^\rho_j(R/\fq_i)$ for all $i+j\le t$. Hence, letting $e$ be such that $p^e \ge 2^tn_0$, we have $J^{\fbp{p^e}}$ annihilates $H^j_\fm(R/\fq_i)/0^\rho_j(R/\fq_i)$. In particular, for $i +j \le t$, \[\fq^{\fbp{p^e}} H^j_\fm\left(R/\fq_i^{\fbp{p^e}}\right)\subset 0^\rho_j\left(R/\fq_i^{\fbp{p^e}}\right).\] As in \cite[Prop. 3.3]{HQ22}, $\HSL_0^\rho(R/\fq_i) \le e_0 +  \sum_{k=0}^t \binom{t}{k} \HSL_k(R) =: e_1$.

The remainder of the proof follows that of \cite[Thm. 3.6]{HQ22}, but we reproduce their argument. Note that $0^\rho_0(R/\fq)$ is potentially slightly smaller than expected, \cite[Lem. 2.9]{HQ22} shows that $0^\rho_0(R/\fq) = \fq^F\cap (\fq:\fm^{\infty})/\fq$. Suppose $x \in \fq^F \setminus \fq$. When $\dim R/(\fq:_Rx) = 0$, $x \in \fq:\fm^{\infty}$ and so $$x+\fq \in \fq^F\cap (\fq:_R\fm^{\infty})/\fq = 0^\rho_0(R/\fq).$$ Thus $x^{p^{e_1}} \in \fq^{\fbp{p^{e_1}}}$. Hence, it suffices to find uniform $e_2$, so that $x^{p^{e_2}} \in \left(\fq^{\fbp{p^{e_2}}}:_R\fm^{\infty}\right)$.

Suppose $\dim R/(\fq:_R x) = s$. We verify, for $h=\HSL R$, that one has $\dim R/\left( \fq^{\fbp{p^{hs}}}:_R x^{p^{hs}}\right) = 0$. When $s=0$ the statement follows by definition so suppose $s>0$. Let $\fp$ be a prime minimal over $(\fq:_Rx)$ with $R/\fp=s$. By minimality, $x/1 \in (\fq R_\fp:_{R_\fp}(\fp R_\fp)^\infty)$. As Frobenius closure localizes, we see $x/1 \in (\fq R_\fp)^F$, so $x/1 \in 0^\rho_0(R_\fp/\fq R_\fp)$. Since $s>0$, we have $\fp \neq \fm$ so $\fq$ is generated by a regular sequence in $R_\fp$, and thus, by \cite[Rmk. 3.4 and Lem. 3.5]{HQ22}, $\HSL^\rho_0(R_\fp/\fq R_\fp) \le \HSL(H^t_{\fp R_\fp}(R_\fp)) \le h$. Consequently, $x^{p^h}/1 \in \fq^{\fbp{p^h}}R_\fp$. Thus we have $\dim R/\left(\fq^{\fbp{p^h}}:_R x^{p^h}\right) \le s-1$ as $R_\fp / \left(\left(\fq^{\fbp{p^h}}\right) R_\fp :_{R_\fp} x^{p^h}/1\right) = 0$ for all $\fp \in \Min R/(\fq:_Rx)$ with $\dim R/\fp = s$. 

Now as $\dim R/\left(\fq^{\fbp{p^{hs}}}:_R x^{p^{hs}}\right) = 0$ we obtain a uniform bound since $s \le d-t$ always. Thus, setting $e_2 := (d-t)h$, for $x \in \fq^F$, $x^{p^{e_2}} \in \fq^{\fbp{p^{e_2}}} :_R \fm^{\infty}$, and hence $x^{p^{e_2+e_1}} \in \fq^{\fbp{p^{e_2+e_1}}}$, i.e. $\fte \fq \le e_2 + e_1 = (d-t)h+e_1$ for all $\fq \in \fQ_{J,t}$.
\end{proof}

Now, we can specialize to the case that $J$ is inside $\fb_0(R) \cap \cdots \cap \fb_{t-1}(R)$, since this is the largest ideal $I$ with respect to which $\gfdp_I(R) \ge t$.

\begin{corollary}
Let $(R,\fm)$ be a local ring of dimension $d$. There is an $e_0=e_0(t)$ such that, for any filter regular sequence $x_1,\ldots,x_t$ inside $\fb_0(R) \cap \cdots \cap \fb_{t-1}(R)$, $\fte (x_1,\ldots,x_t) \le e_0$.
\end{corollary}

\bibliographystyle{amsplain}
\bibliography{ref.bib}
\end{document}